\documentclass[11pt]{article}
\usepackage{amssymb,amsmath,graphicx,amsfonts}
\usepackage{amsmath}
\usepackage{amsfonts}
\usepackage{tikz}
\usetikzlibrary{arrows}
\usepackage{color}

\newtheorem{theorem}{Theorem}[section]

\newtheorem{lemma}[theorem]{Lemma}
\newtheorem{proposition}[theorem]{Proposition}
\newtheorem{corollary}[theorem]{Corollary} 
\newtheorem{remark}[theorem]{Remark}
\newtheorem{example}[theorem]{Example}
\newtheorem{problem}{Problem}

\newtheorem{notation}[theorem]{Notation}
\newtheorem{definition}[theorem]{Definition}

\newtheorem{precor}{{\bf Corollary}}

\newtheorem{precon}{{\bf Conjecture}}

\newtheorem{predefin}{{\bf Definition}}

\newtheorem{preexm}{{\bf Example}}

\newtheorem{preappl}{{\bf Application}}

\newtheorem{prelem}{{\bf Lemma}}

\newtheorem{preproof}{{\bf Proof.\ }}

\newenvironment{proof}[1]{\begin{preproof}{\rm
               #1}\hfill{$\blacksquare$}}{\end{preproof}}
\newtheorem{presproof}{{\bf Sketch of Proof.\ }}

\newtheorem{prethm}{{\bf Theorem}}

\newtheorem{preconj}{{\bf Conjecture}}

\newtheorem{preques}{{\bf Question}}

\newtheorem{prealphthm}{{\bf Theorem}}

\newtheorem{prepro}{{\bf Proposition}}

\newtheorem{preprb}{{\bf Problem}}

\newtheorem{prerem}{{\bf Remark}}

\def\mult{{\rm mult}}
\def\conct[#1,#2]{\mbox {${#1} \leftrightarrow {#2}$}}
\def\dconct[#1,#2]{\mbox {${#1} \rightarrow {#2}$}}
\def\deg[#1,#2]{\mbox {$d_{_{#1}}(#2)$}}
\def\mindeg[#1]{\mbox {$\delta_{_{#1}}$}}
\def\maxdeg[#1]{\mbox {$\Delta_{_{#1}}$}}
\def\outdeg[#1,#2]{\mbox {$d_{_{#1}}^{^+}(#2)$}}
\def\minoutdeg[#1]{\mbox {$\delta_{_{#1}}^{^+}$}}
\def\maxoutdeg[#1]{\mbox {$\Delta_{_{#1}}^{^+}$}}
\def\indeg[#1,#2]{\mbox {$d_{_{#1}}^{^-}(#2)$}}
\def\minindeg[#1]{\mbox {$\delta_{_{#1}}^{^-}$}}
\def\maxindeg[#1]{\mbox {$\Delta_{_{#1}}^{^-}$}}

\def\dre[#1,#2,#3]{\mbox {${\cal E}_{_{#3}}(#1,#2)$}}
\def\pdre[#1,#2,#3]{\mbox {${\cal P}_{_{#3}}(#1,#2)$}}
\def\var[#1,#2]{\mbox {${\rm Var}_{_{#1}}(#2)$}}
\def\ls[#1]{\mbox {$\xi^{^{#1}}$}}
\def\hom[#1,#2]{\mbox {${\rm Hom}({#1},{#2})$}}
\def\onvhom[#1,#2]{\mbox {${\rm Hom^{v}}(#1,#2)$}}
\def\onehom[#1,#2]{\mbox {${\rm Hom^{e}}(#1,#2)$}}
\def\core[#1]{\mbox {$#1^{^{\bullet}}$}}
\def\cay[#1,#2]{\mbox {${\rm Cay}({#1},{#2})$}}

\def\dCay{\overrightarrow{{\rm Cay}}}
\def\cays[#1,#2]{\mbox {${\rm Cay_{s}}({#1},{#2})$}}
\def\dirc[#1]{\mbox {$\stackrel{\rightarrow}{C}_{_{#1}}$}}
\def\cycl[#1]{\mbox {${\bf Z}_{_{#1}}$}}

\def\sdg[#1]{\mbox {$\stackrel{\leftrightarrow}{#1}$}}
\def\Ga{\Gamma}
\def\2sc{{{\rm 2S}}}
\def\dsc{{\rm \overrightarrow{2S}}}
\def\da{{\rm \overrightarrow{DA}}}
\def\Aut{{\rm Aut}}
\def\Alt{{\rm Alt}}
\def\Sym{{\rm Sym}}

\def\la{\langle}
\def\ra{\rangle}
\def\da{{\rm DA}}
\def\dda{{\rm \overrightarrow{DA}}}

\def\dga{{\rm \overrightarrow{GA}}}
\def\ca{{\rm CA}}

\newcommand{\Der}{\mathop{\mathrm{Der}}}

\begin{document}
\begin{center}
{\Large \bf Derangement action digraphs and graphs}\\
\vspace*{0.5cm}
{\bf Moharram N. Iradmusa$^{a}$, Cheryl E. Praeger$^b$}\\
\vspace*{0.2cm}
{\it {\small $^{a}$Department of Mathematical Sciences, Shahid Beheshti University, G.C.\\
P.O. Box 19839-63113, Tehran, Iran\\
m\_iradmusa@sbu.ac.ir\\
$^{b}$Centre for the Mathematics of Symmetry and Computation, The University\\
 of Western Australia, 35 Stirling Highway, Crawley, WA 6009, Australia\\
cheryl.praeger@uwa.edu.au}}\\
\end{center}
\begin{abstract}
A derangement of a set $X$ is a fixed-point-free permutation of $X$.
Derangement action digraphs are closely related to group action digraphs introduced by Annexstein, Baumslag and Rosenberg in 1990. 
For a non-empty set $X$ and a non-empty subset $S$ of derangements of $X$, the derangement action digraph $\dda(X, S)$ has 
vertex set $X$, and an arc from $x$ to $y$ if and only if $y$ is the image of $x$ under the action of some element of $S$, so by
definition it is a simple digraph. 
In common with Cayley graphs and Cayley digraphs, derangement action digraphs may be useful to model networks since the same 
routing and communication schemes can be implemented at each vertex. We prove that the family of derangement action digraphs
contains all Cayley digraphs, all finite vertex-transitive simple graphs, and all finite regular simple graphs of even valency. 
We determine necessary and sufficient conditions on $S$ 
under which $\dda(X, S)$ may be viewed as a simple undirected graph of valency $|S|$.
We investigate structural and symmetry properties of these digraphs and graphs, pose several open problems, 
and give many examples.\\
\begin{itemize}
\item[]{{\footnotesize {\bf Key words:}\ Cayley graph, vertex-transitive graph, group action digraphs, derangements.}}
\item[]{ {\footnotesize {\bf Subject classification: 05C25} .}}
\end{itemize}
\end{abstract}
\section{Introduction}\label{sec:intro} 

Group action digraphs were introduced in 1990  by Annexstein, Baumslag and Rosenberg~\cite{ann90a} as models for  interconnection 
networks underpinning parallel computer architectures. We study a closely related family of digraphs, called \emph{derangement action digraphs}. 
We are particularly concerned with their structural and symmetry properties.  

As  in \cite[Abstract]{ann90a}, for a set $X$ and a subset $S$ of the symmetric group 
$\Sym(X)$ of permutations on $X$, the  \emph{group action digraph} $\dga(X,S)$ is the digraph with vertex set 
$X$ such that, for each $s\in S$ and each vertex 
$v\in X$, there is an arc labelled $s$ from $v$ to the image $v^s$ of $v$ under the action of $s$.
In particular, $\dga(X,S)$ has a loop (an arc from a vetex to itself) at a vertex $v$ whenever 
$v$ is a fixed point of some permutation in $S$. Throughout this paper we wish to avoid loops so we will require 
each arc of  $\dga(X,S)$ to involve two distinct vertices as its `endpoints', equivalently we require each
 $s\in S$ to act without fixed points on $X$. Such a permutation 
$s$ is called a \emph{derangement} of $X$.  Thus we require $S$ to be a subset of the set 
$\Der(X)$ of derangements of $X$. We call such a digraph $\dga(X,S)$ a \emph{loopless group action digraph}.

Loopless group action digraphs may have `multiple arcs', that is to say, there may exist
vertices $v\in X$ and distinct derangements $s, t\in S$ such that  $v^s$ and $v^t$ are equal, say to $u$. In this case 
there are distinct arcs from $v$ to $u$, with one labelled $s$ and the other labelled $t$. 
The number of arcs in  $\dga(X,S)$ from $v$ to $u$ is the number of elements $s\in S$ such that $v^s=u$,
and is called the \emph{multiplicity} of $(v,u)$, denoted $\mult(v,u)$.  If there are no arcs from $v$ to $u$ then we
write $\mult(v,u)=0$.  We say that $\dga(X,S)$ is \emph{multiplicity-free} if $\mult(v,u)\leq 1$ for all
$(v,u)\in X\times X$.  
We are interested in \emph{simple digraphs}. These are digraphs $\Ga= (X, A)$ consisting of a set $X$ of vertices and
a set $A$ of ordered pairs of distinct vertices, called arcs. That is to say, $A$ is a subset of  $X^{(2)} :=\{(u,v) | u,v\in X, u\ne v\}$.
For each loopless group action digraph there is an
underlying simple digraph defined as follows.

\begin{definition}\label{def:dad}
{\rm
Let $X$ be a non-empty set, and let $S\subseteq\Der(X)$, with $S\ne\emptyset$. Define the \emph{derangement action digraph}  
with derangement connection set $S$ as the digraph $\dda(X,S) = (X, A)$ with vertex set $X$ and 
arc set $A= \{ (x,x^s) | x\in X, s\in S\} \subseteq X^{(2)}$. 
}

\end{definition}
 To emphasise:  \emph{we regard $(v,v^s)$ and $(v,v^t)$ as the same arc of $\dda(X,S)$ if  $v^s=v^t$, for $s, t\in S$}. 
 Moreover, since $S\subseteq \Der(X)$, we have  $A\subseteq X^{(2)}$ so that $\dda(X,S)$ is a simple digraph, and  is
determined uniquely by the loopless group action digraph  $\dga(X,S)$. If also  $\dga(X,S)$ is multiplicity-free then  
$\dga(X,S)$ and  $\dda(X,S)$ are `essentially the same' if we ignore the labels on the  arcs of  $\dga(X,S)$. 
It is possible, however, to have  $\dda(X,S)= \dda(X,S')$ for distinct $S, S'\subset\Der(X)$, even if both  $\dga(X,S)$ and
 $\dga(X,S')$ are loopless and multiplicity-free  (Example~\ref{ex:nonuniqueS} with $\{S,S'\}=\{S_1,S_2\}$). 
 This non--uniqueness of $S$ raises a natural question:
 
 \begin{center}
 \emph{If  $\dga(X,S')$ is loopless with multiple arcs, when can we replace $S$ by some $S'\subset\Der(X)$\\ such that   
 $\dga(X,S')$ is multiplicity-free and  $\dda(X,S)= \dda(X,S')$?}
 \end{center}
 \noindent
Sometimes we can do this   (Example~\ref{ex:nonuniqueS} with  $(S,S')=(S_2,S_3)$) but not always   (Example~\ref{ex:irregular} and Remark~\ref{rem:irregular}).  Some concepts used in these examples are introduced in Notation~\ref{notation-intro1}.

\begin{example}\label{ex:nonuniqueS}
{\rm
Let $X=\{1,2,3,4\}$, $S_1=\{(1234), (1432)\}$, $S_2=\{ (12)(34), (14)(23)\}$, and 
$S_3=\{(1234), (12)(34), (14)(23)\}$. Then each subset $S_i\subset \Der(X)$, 
and 
\[
\dda(X,S_1)=\dda(X,S_2)=\dda(X,S_3)
\] 
is a simple graph, namely an undirected cycle $C_4$ of length $4$.  In fact  $S_2$ is the only proper subset of $S_3$ 
giving the same simple digraph  $\dda(X, S_3)$. In addition $\dga(X,S_1)$ and $\dga(X,S_2)$  are multiplicity-free
but $\dga(X,S_3)$ is not. 
%
}
\end{example}

\begin{example}\label{ex:irregular}
{\rm
Let $X=\{1,2,3,4,5,6,7,8\}$ and $S=\{a, b, c\}$ where $a = (12)(34)(56)(78)$, $b=(18)(27)(34)(56)$, 
and $c=(18)(45)(23)(67)$. 
Then $S\subset \Der(X)$, and some vertex pairs $(u,v)$ in $\dga(X,S)$  have $\mult(u,v)=2$, 
but also $\dda(X,S)$ is a graph, namely a cycle 
$(1,2,3,4,5,6,7,8)$ of length $8$ with one extra edge between the vertices $2$ and $7$. 
}
\end{example}

\begin{remark}\label{rem:irregular}{\rm 
The digraph $\dda(X,S)$ in Example~\ref{ex:irregular}  is not regular since six of the vertices have 
valency $2$, while the vertices $2$ and $7$ have valency $3$. 
Also omitting any of the derangements $a, b, c$ 
from $S$ produces a different derangement action digraph. So even though $\mult(x,y)=2$ for some pairs $(x,y)$,
no proper subset of $S$ will produce the same digraph $\dda(X,S)$ with smaller 
multiplicities. In fact there is no derangement set $S'\ne S$ such that 
$\dda(X, S')=\dda(X,S)$, and  such that $\dga(X, S')$ has smaller  multiplicities 
than $\dga(X, S)$. To see this, observe that 
 $\dda(X,S)$ has $9$ edges, and hence $18$ arcs, and each derangement in $S'$ 
would correspond to exactly $8$ arcs. So $|S'|\geq 3$, and $S'$ can do no better than the subset $S$ which produces 
$12$ arcs with multiplicity $1$, and $6$ arcs with multiplicity $2$. 
}
\end{remark}

It would be interesting to have necessary and sufficient conditions on $S$ for such a replacement to be possible, and currently this is an open problem. 
However we have a simple criteria for  $\dga(X,S)$ to be multiplicity-free, and it allows us to show that the family of
derangement action digraphs is large. 
 We always assume that the subset $S$ is finite, and some of our proofs require also that the set $X$ is finite. If a result does require $X$ to be finite we will specify this in the statement. For convenience we record in Notation~\ref{notation-intro1} below the concepts and notation used in Theorem~\ref{thm:dareg1} (which is proved in Sections~\ref{sec:props} and~\ref{sec:graph}). 

\begin{theorem}\label{thm:dareg1}
\begin{enumerate}
\item[(a)] Let $S\subseteq \Der(X)$ with $X$ non-empty and  $S$ finite.  Then  $\dga(X,S)$ is multiplicity-free $\iff$ $SS^{-1}\subseteq \Der(X)\cup\{1\}$  $\iff$ 
 $\dda(X,S)$ is regular of valency  $|S|$.

\item[(b)]  Every finite simple regular digraph is a derangement action digraph.
\end{enumerate}
\end{theorem}

The family of finite derangement action digraphs is larger than the class of  finite simple regular digraphs: it certainly contains some digraphs that are not regular (Example~\ref{ex:irregular}).

%

\begin{notation}\label{notation-intro1}{\rm 
For a simple digraph $\Ga=(X,A)$ and arc $(x,y)\in A$, 
$y$ is called an \emph{out-neighbour} of $x$ and $x$ is called  an \emph{in-neighbour} of $y$.  
The number of out-neighbours (respectively in-neighbours) of $x$ is called the \emph{out-valency}
(respectively,  \emph{in-valency}) of $x$.  If there is a constant $k$ such that each vertex has out-valency $k$ 
and in-valency $k$, then $\Ga$ is said to be \emph{$k$-regular}, or simply \emph{regular}, and $k$ is called the \emph{valency} of $\Ga$. 
If $A$ contains both $(x,y)$ and $(y,x)$ then we call the 
unordered pair $\{(x,y), (y,x)\}$ an \emph{edge} of $\Ga$.
If $A$ is \emph{symmetric}, that is, if $(y,x)\in A$ implies $(x,y)\in A$, then we call
$\Ga$ a \emph{simple  graph}, and often interpret it as the \emph{simple undirected graph} $(V, E)$, where
$E$ is the set of edges. 
For a subset $S\subseteq \Sym(X)$, we write $S^{-1}=\{g^{-1}\mid g\in S\}$, 
and if $S=S^{-1}$ we say that $S$ is \emph{self-inverse}. For subsets $A, B\subset\Sym(X)$,
we write $AB=\{a b\mid a\in A, b\in B\}$. 
}
\end{notation}

If a  derangement action digraph $\dda(X,S)$ is a graph, then $\dda(X,S)$ 
is called a  \emph{derangement action graph}, and is sometimes denoted $\da(X,S)$.
We introduce a condition on $S$ which turns out to be sufficient for
$\dda(X,S)$ to be a regular simple graph.

\begin{definition}\label{def:cds}
{\rm
Let $X$ be a non-empty set. A non-empty subset $S\subseteq\Der(X)$
is said to be  \emph{closed} if  both of the following conditions hold.
\begin{center}
(1)\ $x^S=x^{S^{-1}}$ for each $x\in X$; \quad and \quad (2) $SS^{-1}\subseteq \Der(X)\cup\{1\}$. 
\end{center} 
}
\end{definition}

\begin{theorem}\label{thm:dasimple}
Let $X$ be a non-empty set, and let $S\subseteq \Der(X)$ with $S$ finite. Then $\dda(X,S)$ is a derangement 
action graph which is regular of valency $|S|$ if and only if $S$ is  closed.
\end{theorem}

We do not, however, know if $\dda(X,S)$ can be a regular graph of valency less than $|S|$.

\begin{problem}\label{q3} 
Decide whether or not $\dda(X,S)$ can be a regular graph of valency less than $|S|$, and 
if so, find necessary and sufficient conditions on $S$ for this to occur.
\end{problem}


Theorem~\ref{thm:dasimple} is proved in Section~\ref{sec:props}. 
If the set $S\subseteq\Der(X)$ is self-inverse, then condition 
$(1)$  of  Definition~\ref{def:cds} is vacuously true,  and hence $S$ is closed if and only if  $S^2\subseteq \Der(X)\cup\{1\}$. 
However not all closed derangement sets have this property, see Example~\ref{ex:derclosednonCayley}.

\begin{example}\label{ex:derclosednonCayley}{\rm
Let $X=\{1,2,3,4,5,6\}$ and 
\[
S = \{ (1,2,3,4,5,6),(1,3,2)(4,6,5), (1,6,4,3)(2,5) \} \subseteq \Der(X).
\] 
Here $S$ is closed but not self-inverse, and $\dda(X,S)$ is the regular graph 
of valency $3$ shown in Figure~\ref{pic1}. We note that also $\dda(X,S)=\dda(X,S')$
for a different closed self-inverse subset:
\[
S' = \{ (1,2,3,4,5,6), (1,6,5,4,3,2), (1,3)(2,5)(4,6) \} \subseteq \Der(X).
\] 
}
\end{example}

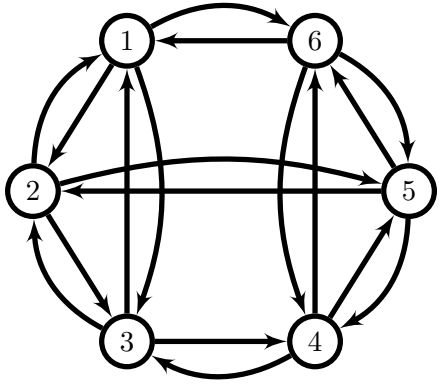
\begin{figure}
\begin{center}
\begin{tikzpicture}[scale=0.50]
\tikzset{vertex/.style = {shape=circle,draw, line width=2pt,opacity=1.0}}
\tikzset{edge/.style = {->,> = latex', line width=2pt,opacity=1.0}}
\node[vertex] (a) at  (-5.0,0) {2};
\node[vertex] (b) at  (-2.5,4) {1};
\node[vertex] (c) at  (2.5,4) {6};
\node[vertex] (d) at  (5.0,0) {5};
\node[vertex] (e) at (2.5,-4) {4};
\node[vertex] (f) at (-2.5,-4) {3};
\draw[edge] (a)  to[bend left] (b);
\draw[edge] (b) to (a);
\draw[edge] (b)  to[bend left] (c);
\draw[edge] (c) to (b);
\draw[edge] (c)  to[bend left] (d);
\draw[edge] (d) to (c);
\draw[edge] (d)  to[bend left] (e);
\draw[edge] (e) to (d);
\draw[edge] (e)  to[bend left] (f);
\draw[edge] (f) to (e);
\draw[edge] (f)  to[bend left] (a);
\draw[edge] (a) to (f);
\draw[edge] (a)  to[out=15,in=165] (d);
\draw[edge] (d) to (a);
\draw[edge] (b) to[out=-70,in=70] (f);
\draw[edge] (f) to (b);
\draw[edge] (c)  to[out=250,in=110] (e);
\draw[edge] (e) to (c);
\end{tikzpicture}
\caption{The derangement action graph $\da(X,S)=\da(X,S')$ of Example~\ref{ex:derclosednonCayley}.}
\label{pic1}
\end{center}
\end{figure}

Even if we restrict the derangement subsets to be closed and self-inverse, we still obtain
a large class of derangement action graphs, (see Remark~\ref{rem:perfectmatching} for a definition and discussion of perfect matchings).

\begin{theorem}\label{thm:vtr} 
The family of regular simple graphs which can be expressed as $\da(X,S)$, for some closed, self-inverse, subset $S\subseteq\Der(X)$, contains all of the following:
\begin{enumerate}
\item[(a)]   every finite regular simple graph of even valency;
\item[(b)]   every finite regular simple graph of odd valency that has a perfect matching;
\item[(c)]   every finite vertex-transitive graph, and every finite regular bipartite simple graph. 
\end{enumerate}
\end{theorem}

\begin{remark}\label{rem:perfectmatching}{\rm 
(a) A \emph{perfect matching} in a finite simple graph $\Ga = (V, A)$  is a subgraph $\Ga' = (V, A')$ of $\Ga$
such that each connected component of $\Ga'$ is an edge, that is, its arc set is of the form $\{(u,v),(v,u)\}$.   Tutte's `$1$-factor theorem' gives a necessary and sufficient condition for a  finite simple graph $(V,A)$ to have a perfect matching: for each subset $U\subseteq V$, the subgraph induced on $V\setminus U$ has at most $|U|$ 
connected components with an odd number of vertices (see \cite[Theorem 16.13, p. 430]{bon08a}  or \cite[p. 84]{lp86}). In particular, each finite vertex transitive graph with an even number of vertices, and each finite  regular bipartite simple graph has a perfect matching (see for example, \cite[Theorem 3.5.1]{god01a} and \cite[Theorem 3]{np81}, respectively). Thus every finite vertex-transitive simple graph, and every finite regular bipartite simple graph is of the form $\da(X,S)$ for some  closed, self-inverse, subset $S$ of $\Der(X)$.

\medskip
(b)  The condition of having a perfect matching in Theorem~\ref{thm:vtr}(b) is essential, that is to say, 
a finite regular simple graph of odd valency can be expressed as $\da(X,S)$, for some closed, self-inverse, subset $S\subseteq\Der(X)$,
if and only if it has a perfect matching. To see this, suppose that 
$\Ga=(X,S)$ is a finite regular graph of odd valency, and  $\Ga=\da(X,S)$ for some  closed, 
self-inverse, subset $S$ of $\Der(X)$. Then, by Theorem~\ref{thm:dasimple},  $\Ga$ has valency $|S|$, and hence $|S|$ is odd. 
Since $S$ is self-inverse, it follows that there exists $s\in S$ such that $s=s^{-1}$. This implies that $s^2=1$, and the subgraph 
$\da(X,\{s\})$ of $\Ga$ is a perfect matching of $\Ga$. 

\medskip
(c) It would be interesting to know the extent to which finiteness could be removed from  
Theorem~\ref{thm:vtr}. For example, if $\Ga=(X,A)$ is an infinite simple $k$-regular vertex-transitive graph, then by \cite[Theorem 1.2]{CL}, such a $k$-regular graph  $\Ga=(X,A)$ has a perfect matching  
$(X,A_1)$. Defining  $s_1: X\rightarrow X$ by $x^{s_1}=y$ if and only if $(x,y)\in A_1$, we see that $s_1$ is an involution in $\Der(X)$, and $(X,A_1) =\da(X,\{s_1\})$.  If $k$ were finite and $(X,A_1)$ were left invariant by a vertex-transitive subgroup of automorphisms of $\Ga$, then $(X,A\setminus A_1)$ would be a $(k-1)$-regular vertex-transitive graph and perhaps an inductive argument might be applied. For example such an argument works for Cayley digraphs (see Subsection~\ref{sec:cayley}).

\begin{problem}\label{qvtr} 
Decide which infinite regular simple graphs of finite valency are derangement action graphs. 
\end{problem}

(d)  Theorem~\ref{thm:vtr} is proved in Section~\ref{sec:graph}.
We note that there are also some non-regular graphs of the form $\da(X,S)$ with $S\subseteq\Der(X)$ and $S$ self-inverse (but of course $S$ not closed by Theorem~\ref{thm:dasimple}), see Example~\ref{ex:irregular}. 
}
\end{remark}


\subsection{Derangement action digraphs and two-sided group digraphs}\label{sec:cayley}

In developing a theory of derangement action digraphs we were motivated by a desire for 
a natural family of simple digraphs which properly contains all finite simple vertex transitive  graphs and digraphs,
but is not `too large'. We believe that the family of derangement action digraphs has
these desirable properties (Theorems~\ref{thm:dareg1} and~\ref{thm:vtr}). 

Since many vertex-transitive simple digraphs are Cayley digraphs,
we focused first on constructions generalising the Cayley digraph construction, 
which takes as input a group $G$ and a subset $S\subseteq G\setminus \{1\}$, and 
produces a vertex-transitive simple digraph.  
The \emph{Cayley digraph} 
$\dCay(G,S)$ has vertex set $G$ and, for each $s\in S$ and each 
group element $g\in G$, an arc $(g,sg)$ from $g$ to the image 
$sg$ of $g$ under the left multiplication map $\hat s: g\mapsto sg$. Since $s\ne 1$ 
the map $\hat s\in\Der(G)$ and hence, for $\hat S :=\{ \hat s\mid s\in S\}$, we see that 
$\dCay(G,S)$ is equal to the derangement action digraph $\dda(G,\hat S)$ 
associated with the group action digraph $\dga(G,\hat S)$. Moreover $\dCay(G,S)$
is vertex-transitive since each right multiplication map 
$g\mapsto gh$ is an automorphism of $\dCay(G,S)$. 

Cayley digraphs also play a role in the theory of group action digraphs described in 
\cite{ann90a}: we describe this role briefly  in Subsection~\ref{background}, noting 
the differences with our approach.
Our first attempt to generalise Cayley digraphs was made in \cite{ira01}, where we introduced the family of
two-sided group digraphs. The family of two-sided group digraphs also contains all Cayley digraphs \cite[Proposition 1.2]{ira01},
but fails to contain all finite vertex-transitive graphs, not even the Petersen graph \cite[Theorem 1.14]{ira01}. 
On the other hand not all two-sided group digraphs are derangement action 
digraphs as we see below. Despite this fact, the intersection of 
these two families is sufficiently rich to provide  numerous insightful examples.

For a group $G$ and non-empty subsets  $L, R\subseteq G$, the \emph{two-sided group digraph}
$\dsc(G;L,R)$ is the simple digraph with vertex set $G$ such that $(x,y)$ is 
 an arc if and only if $y=\ell^{-1}xr$ for some $\ell\in L$ and $r\in R$. In particular,
 each $\dCay(G,S) = \dsc(G;S^{-1},\{1\})$.  Setting
 ${S}(L,R) = \{ \lambda_{\ell,r}\mid \ell\in L, r\in R\}$, where 
\begin{equation}\label{eq:lambda} 
\lambda_{\ell,r}: g\mapsto \ell^{-1}gr\quad \mbox{for}\ g\in G, 
\end{equation}
we have ${S}(L,R)\subseteq \Sym(G)$ so we can form  the group action digraph
$\dga(G,S(L,R))$. The digraph $\dga(G,S(L,R))$ is loopless, equivalently $S(L,R)\subseteq\Der(G)$, 
if and only if for each $\ell\in L, r\in R$, the $G$-conjugacy classes containing $\ell$ and $r$ are distinct,  
\cite[Lemma 3.1(b)]{ira01}. When this property holds, $\dsc(G;L,R)$ is equal to the derangement 
action digraph $\dda(G,S(L,R))$, and an analogue of Theorem~\ref{thm:dasimple} is proved 
for these digraphs in \cite[Theorem 1.5]{ira01}: the property that $S(L,R)$ is a closed subset of 
derangements is equivalent to the `$2S$-graph-property' of \cite[Definition 1.4]{ira01} (and in this case
$|S(L,R)| = |L|\cdot|R|$). 

Two-sided group digraphs and graphs are a convenient source of informative examples. 
For instance, if $S\subseteq\Der(X)$ is closed then $\dda(X,S)$ has constant in-valency 
and constant out-valency (see Lemma~\ref{lem:damultfree}). However $\dda(X,S)$ may have 
 constant out-valency, but non-constant in-valency (even if $\dga(X,S)$ is loopless and multiplicity-free). 
 The following lovely example, which 
 is  a two-sided group digraph, appeared in \cite{CKLSSY}. 

\begin{example}\label{ex:inout}
{\rm  \cite[Example 2.5]{CKLSSY}\quad 
Let $X=\Alt(4)$, the alternating group of order $12$,  and let $S = S(L,R)$ where 
$L = \{ 1, (243)\}$ and $R=\{ (234), (12)(34), (132), (14)(23)\}$. An easy computation shows that 
$S(L,R)\subset\Der(X)$ and $|S|=|L|\cdot|R|=8$. As discussed in \cite{CKLSSY}, especially  in \cite[Figure 1]{CKLSSY},
$\dda(X,S)=\dsc(S;L,R)$  has constant out-valency $7$, 
while half of the vertices have in-valency $6$ and the other 
half have in-valency $8$.  
}
\end{example}


\begin{problem}\label{q2} 
Find necessary and sufficient conditions such that, if $\dda(X,S)$ has constant 
out-valency,  then it must also have constant in-valency.
\end{problem}

\subsection{Connectivity of derangement action digraphs}

For vertices $x, y$ of a simple digraph $\Ga=(X,A)$ we say that $x$ \emph{is connected to} 
$y$ if either $x=y$ or there is a vertex sequence $v_0, v_1,\dots,v_r$
such that $v_0=x$, $v_r=y,$ and $(v_{i-1},v_i)\in A$ for each $i=1,\dots,r$. If this
connectivity relation is an equivalence relation then the sub-digraph induced on an equivalence class
is a connected simple digraph and is called a \emph{connected component} of $\Ga$.  
It turns out, in particular, that the connectivity relation is an equivalence relation for each finite 
derangement action digraph, and that the connected components are 
also derangement action digraphs.  For a subset $S\subseteq\Sym(X)$ let  $\langle S\rangle$ 
denote the subgroup of $\Sym(X)$ generated by $S$.

\begin{theorem}\label{thm:connectivity1}
Let $X$ be a non-empty set, and let $S\subseteq\Der(X)$. Then the following hold.
\begin{enumerate}
\item[(a)] If  each $s\in S$ has no infinite cycles, then the connectivity relation for  $\dda(X,S)$ is an equivalence 
relation on $X$.
\item[(b)] If the connectivity relation for  $\dda(X,S)$ is an equivalence 
relation, then the equivalence classes are the orbits of $\langle S\rangle$.
Moreover, for each $\langle S\rangle$-orbit $X'$, the restriction $S' := S|_{X'}\subseteq 
\Der(X')$, and the connected component on $X'$ is $\dda(X',S')$. 
\item[(c)] If $\Ga$ is a simple digraph for which the connectivity relation 
is an equivalence relation, then $\Ga$ is a derangement action digraph if and only if each 
connected component of $\Ga$ is a derangement action digraph.
\end{enumerate}
\end{theorem}

Theorem~\ref{thm:connectivity1} is proved in Section~\ref{sec:conn}. The following example
shows that the connected components in Theorem~\ref{thm:connectivity1} (c) may be 
non-isomorphic, and even have different sizes.


\begin{example}\label{ex4}{\rm
The simple graph $\da(\mathbb{Z}_7,S)$, where $S=\{(012)(3456),(021)(3654)\}$,  has one component 
isomorphic to the cycle $C_3$, and one component isomorphic to $C_4$. 
}
\end{example}

If $\Ga=\dda(X,S)$ satisfies the hypotheses of  Theorem~\ref{thm:connectivity1} (b) with $\langle S\rangle$ transitive on $X$, 
it is tempting to examine quotient digraphs of $\Ga$ modulo $\langle S\rangle$-invariant partitions of $X$.
However although elements of $S$  permute the parts of these partitions they do not in general 
induce derangements, so the family of derangement action digraphs is not closed under forming such 
quotients. Nevertheless the class is closed under several graph product constructions 
and these are discussed  in Section~\ref{sec:conn}. 

\subsection{Isomorphisms of derangement action digraphs}

Finally we make a few comments about isomorphisms of derangement action digraphs.
An \emph{isomorphism} between two simple digraphs is a bijection from the vertex set of the first 
to the vertex set of the second, which 
induces a bijection from the arc set of the first digraph to that of the second.  Without loss of generality we may assume that the two digraphs 
have the same vertex set. Then an isomorphism from $\Ga = (X, A)$ to $\Delta = (X, B)$ 
is an element $g\in\Sym(X)$ that maps the arc set $A$ of $\Ga$ to the arc set $B$ of $\Delta$. 
If $\Ga=\dda(X,S) = (X, A)$ and $\Delta=\dda(X,T) = (X, B)$, with $S, T\subseteq\Der(X)$, and 
$g\in\Sym(X)$, then it is straightforward to check that  
$g$ maps $\Ga$ to $\Ga^g = \dda(X,S^g)=(X,A^g)$ where $S^g = g^{-1}Sg$.
Therefore, if $g$ is an isomorphism from $\Ga$ to $\Delta$, then both $A^g=B$ and $\dda(X, S^g)=\dda(X,T)$. 
However, this does not imply that $S^g=T$, since $S$ and $T$ 
may have different cardinalities (see Example~\ref{ex:nonuniqueS} with  $S^g=S_1, T=S_2$) and even if 
they have the same cardinality their elements may not even be conjugate in $\Sym(X)$ 
(see Example~\ref{ex:nonuniqueS} with $S^g=S_1, T=S_2$). 
The condition $A^g=B$  holds if and only if, for each $x\in X$, 
$g$ maps the set  of out-neighbours of $x$ in $\Ga$, namely  $x^S=\{x^s \mid s\in S\}$, to the set 
 of out-neighbours of $x^g$ in $\Delta$,  namely  $(x^g)^T$. This is equivalent to the condition
$x^{S^g}=x^T$ for each $x\in X$. 
If $\Ga=\Delta$ then an isomorphism is called an \emph{automorphism}, and the set of all automorphisms 
forms the \emph{automorphism group} $\Aut(\Ga)$.  Theorem~\ref{thm:aut} follows immediately from this discussion.

\begin{theorem}\label{thm:aut}
Let  $\Ga=\dda(X,S)$ and $\Delta=\dda(X,T)$, where $S, T\subseteq\Der(X)$, and $g\in\Sym(X)$.
\begin{enumerate}
\item[(a)] Then $g$ is an isomorphism from $\Ga$  to $\Delta$ if and only if, for all $x\in X$, $x^{S^{g}}\ =\ x^T.$
\item[(b)]  $\Aut(\Ga)=\{g\in \Sym(X)\mid \forall\, x\in X, x^S=x^{S^{g}}\}$, and in particular,
$\Aut(\Ga)$ contains  $N_{\Sym(X)}(S) =\{ g\in\Sym(X) \mid S^g=S\}$.
\end{enumerate}
\end{theorem}

The last assertion of part (b)  has a rather weak corollary: $\dda(X,S)$ is vertex-transitive if $N_{\Sym(X)}(S)$
is transitive on $X$.  This condition does sometimes hold, for example if $S$ is a subset of an 
abelian regular subgroup of $\Sym(X)$ (e.g. $S=S_1$ or $S_2$ of Example~\ref{ex:nonuniqueS}),
but it is not a necessary condition for vertex-transitivity (e.g. if $S=S_3$ in 
Example~\ref{ex:nonuniqueS} then $N_{\Sym(X)}(S) = \langle(13)(24)\rangle$).

\subsection{Brief comments on group action digraphs and Cayley digraphs}\label{background}

As discussed in \cite{ira01}
there are many meaningful applications of Cayley digraphs in molecular biology, computer science and coding theory, 
 and because of their symmetry properties, Cayley digraphs are used as models 
for many interconnection networks. 
In the paper \cite{ann90a} Annexstein et al developed `an algebraic setting for studying certain structural and algorithmic properties of the 
interconnection networks that underlie parallel architectures'. Their setting involved two kinds of digraphs and 
links between them, namely group action digraphs and Cayley digraphs.  
They explored a relationship between group action 
digraphs and Cayley digraphs which is 
different from that suggested in the definition of 
a Cayley digraph. They 
observed (in \cite[Lemma 3.1]{ann90a}) that, for a 
connected group action digraph $\dga(X,S)$, the 
subgroup $G:=\la S\ra$ of $\Sym(X)$ generated by 
$S$ is a transitive permutation group on $X$. They 
showed that $\dga(X,S)$ can be identified with a coset 
digraph for $G$ (on the cosets of a stabiliser $G_v$ 
for some fixed $v\in X$), and then identified 
$\dga(X,S)$ with a quotient digraph of $\dCay(G,S)$. 
They called the Cayley digraph  $\dCay(G,S)$ a 
\emph{Cayley regular cover} of $\dga(X,S)$: it is often much larger than $\dga(X,S)$. They then explored connections 
between the structure of $\dga(X,S)$ and its Cayley 
regular cover $\dCay(G,S)$. 

It is not in general straightforward to 
give explicit descriptions of these Cayley regular 
covers. Responding to a question in \cite[Problem 47]{Hey}, such descriptions were given independently in \cite{Bru, TS} for the family of Kautz digraphs, 
a class of digraphs which is well behaved in terms of various useful parameters for communication network design. 
Also the Cayley regular covers of the de Bruijn digraphs are known \cite{ES,MS}.
\emph{By contrast, we are interested in studying directly (not via their Cayley regular covers) the  
derangement action (simple) digraphs 
corresponding to loopless group action digraphs. }

We mention a few other general studies related to group action digraphs.  Malni\v{c}~\cite[Section 3]{Mal} 
generalises group action graphs to a class of objects which he calls action graphs, in which the connection set $S$ 
is a non-empty subset of a group $G$ acting on the vertex set $X$ (possibly unfaithfully). He requires $S=S^{-1}$, and 
his concept also allows loops, multiple arcs and so-called `semiedges'. Pisanski et al \cite{PTZ} investigate objects 
which they also call action graphs, but which are essentially group action graphs $\dga(X,S)$ for which the connection subset
$S=S^{-1}$, and $S$ is 
contained in a group acting on $X$. 


%


\section{Proofs of Theorems~\ref{thm:dareg1}(a) and~\ref{thm:dasimple}}\label{sec:props}
%
 We begin by exploring the significance of each of the conditions in Definition~\ref{def:cds}. 
 Recall that a group action digraph $\dga(X,S)$ is loop-less
 if and only if $S\subseteq\Der(X)$. 
 Also $\dga(X,S)$ is multiplicity-free if, for each arc $(x,y)$ of the underlying simple digraph $\dda(X,S)$, there is a unique $s\in S$ such that $y=x^s$; 
  in this case  $\dga(X,S)$ can be identified with $\dda(X,S)$ if we disregard labels on arcs. 
 The first lemma establishes Theorem~\ref{thm:dareg1}(a) and a little bit more.
 
\begin{lemma}\label{lem:damultfree}
Let $X$ be a non-empty set and $S\subseteq \Der(X)$ with $S$ finite. 
Then the following are equivalent:
\begin{enumerate}
\item[(i)]  $\dga(X,S)$ is multiplicity-free;
\item[(ii)] $SS^{-1}\subseteq \Der(X)\cup\{1\}$;
\item[(iii)] each vertex $x$ has out-valency $|S|$ in  $\dda(X,S)$;
\item[(iv)]  each vertex $x$ has  in-valency $|S|$ in  $\dda(X,S)$;
\item[(v)]   $\dda(X,S)$ is regular of valency  $|S|$.
\end{enumerate}
In particular the conclusions of Theorem~$\ref{thm:dareg1}~(a)$ are valid.
\end{lemma}

\begin{proof}{
First we prove (ii) $\Leftrightarrow$ (iii). Suppose that 
$SS^{-1}\subseteq \Der(X)\cup\{1\}$, and let $x\in X$. Then $(x,y)$ is an arc of $\dda(X,S)$
if and only if $y=x^s$ for some $s\in S$. If $y=x^{s_1}=x^{s_2}$ for distinct $s_1, s_2\in S$, then 
$x^{s_1s_2^{-1}}=x$ and $1\ne s_1s_2^{-1}\in SS^{-1}$, which is a contradiction. Thus (iii) holds.
Conversely suppose that  each vertex has out-valency $|S|$ in  $\dda(X,S)$, and consider distinct $s_1, s_2\in S$. 
Then for each $x\in X$,  the vertices $x^{s_1}$ and $x^{s_2}$ are distinct out-neighbours of $x$
in  $\dda(X,S)$.  Hence $x^{s_1s_2^{-1}}\ne x$. Since this holds for each $x\in X$, we have  $s_1s_2^{-1}\in\Der(X)$.
Thus (ii) holds. 

Since inverses of derangements are derangements, it follows that 
$SS^{-1}\subseteq \Der(X)\cup\{1\}$ if and only if $S^{-1}S\subseteq \Der(X)\cup\{1\}$,
and an analogous argument shows that (ii) $\Leftrightarrow$ (iv). 
Also, the condition for $\dga(X,S)$ to  be multiplicity-free is equivalent to condition (iii).
Thus the four conditions (i), (ii), (iii), (iv) are pairwise equivalent. Finally, since conditions (iii) and (iv) are equivalent,
and the combination of (iii) and (iv) is equivalent to  $\dda(X,S)$ being regular of valency  $|S|$,
we see that condition (v) is equivalent to each of the other conditions. 

Theorem~\ref{thm:dareg1}~(a) is the assertion that conditions (i), (ii) and (v) are equivalent, and hence is proved.
}
\end{proof}

The next lemma establishes Theorem~\ref{thm:dasimple}. Recall that a subset 
$A\subseteq X\times X$ is symmetric if $(x,y)\in A$ implies $(y,x)\in A$.

\begin{lemma}\label{lem:dasimple}
\begin{enumerate}
\item[(a)]
Let $X$ be a non-empty set, $S\subseteq \Der(X)$ with $S$ finite, and $\Ga = \dda(X, S)=(X, A)$. 
Then $A$ is symmetric if and only if $x^S=x^{S^{-1}}$ for all $x\in X$.
\item[(b)] The conclusion of Theorem~\ref{thm:dasimple} is valid.
\end{enumerate}
\end{lemma}

\begin{proof}{ (a) 
Suppose first that  $A$ is symmetric. Then we have 
the following equivalent conditions on elements $x,y\in X$. 

\begin{center}
$y\in x^S$ \quad
$\iff$ \quad $y=x^{s}$, for some $s\in S$\quad
$\iff$ \quad $(x,y)\in A$\quad
$\iff$ \quad $(y,x)\in A$,
\end{center}

\noindent
and similarly, $(y,x)\in A$ holds if and only if 

\begin{center}
 $x=y^{s}$,  for some $s\in S$\quad $\iff$ \quad $y=x^{s^{-1}}$, for some $s\in S$\quad
$\iff$ \quad $y\in x^{S^{-1}}$.
\end{center}

\noindent
Thus $x^S=x^{S^{-1}}$ for all $x\in X$. Conversely, 
suppose that this condition holds, and that $(x,y)\in A$.
Then $y=x^{s_1}$ for some  $s_1\in S$, and since   $x^S=x^{S^{-1}}$
we also have  $y=x^{s_2^{-1}}$, or equivalently $x=y^{s_2}$ for some $s_2\in S$, so that $(y,x)\in A$ also.
Thus $A$ is symmetric. 

(b) Now we prove Theorem~\ref{thm:dasimple}. Let $\Ga=\dda(X,S)=(X,A)$ for some finite subset $S\subseteq\Der(X)$. 
Suppose first that $\Ga$ is an $|S|$-regular derangement action graph. Then $A$ is symmetric,
so by part (a), condition (1) of Definition~\ref{def:cds} holds. Also condition (2) of Definition~\ref{def:cds}
follows from Lemma~\ref{lem:damultfree}, since $\Ga$ is $|S|$-regular. Thus $S$ is closed. 
Conversely suppose that $S$ is closed. Then  condition (2) of Definition~\ref{def:cds} holds, and so
 $\Ga$ is $|S|$-regular by Lemma~\ref{lem:damultfree}. Also 
 condition (1) of Definition~\ref{def:cds} holds, and so by part (a) of this lemma, $A$ is symmetric. 
 Thus $\Ga$ is an $|S|$-regular derangement action graph.
}
\end{proof}

\section{Proofs of Theorem~\ref{thm:dareg1}(b) and~\ref{thm:vtr}}\label{sec:graph}  

Recall from Notation~\ref{notation-intro1} that an edge of a simple digraph  is an unordered pair of arcs of the form 
$\{(u,v), (v,u)\}$. 
%
A \emph{spanning sub-digraph} of a simple digraph $\Gamma=(V,A)$
is a digraph $(V, A')$ for some subset $A'\subseteq A$, and if $(V,A')$ is a graph 
it is called a  \emph{spanning subgraph} of $\Gamma$.  

A \emph{simple directed cycle} is a  connected $1$-regular simple
digraph with at least two vertices. Thus if  $\Gamma=(V,A)$ is a simple digraph, 
then a $1$-regular spanning sub-digraph of $\Ga$ is a vertex-disjoint union of simple directed cycles.
Also a $1$-regular spanning  subgraph $(V,A')$ of a simple digraph $\Ga$ is a perfect matching,
that is to say $(V,A')$ is a disjoint union of directed cycles, each of length $2$, with each vertex of 
$\Ga$ occuring in exactly one of these cycles.
We will  apply Tutte's 1-factor theorem and 
Petersen's decomposition theorem~\cite{MR1554815} which address the existence of 
perfect matchings, and $2$-regular spanning subgraphs of finite simple digraphs. 

\begin{lemma}\label{lem:dareg}
Let $k$ be a positive integer and let $\Ga$ be a finite, simple, $k$-regular digraph. 
Then $\Ga$ has a   $1$-regular spanning sub-digraph.
\end{lemma}

\begin{proof}
{We construct from $\Ga$ a simple bipartite graph $\Delta$ with bipartition $(A,B)$ as follows. 
For each $v\in V(\Ga)$, define vertices $v_1\in A$ and $v_2\in B$, and for each arc $(u,v)\in A(\Ga)$, 
define an edge $\{(u_1,v_2), (v_2,u_1)\}\in E(\Delta)$. By assumption, $\Gamma$ is a simple $k$-regular digraph,
and hence by definition, $\Delta$ is a simple  $k$-regular 
graph which is bipartite with bipartition $(A,B)$. Hence, by \cite[Corollary 16.6]{bon08a}, $\Delta$ has a perfect 
matching corresponding to a set of $M$ edges of $\Delta$.  
Define a subset $S$ of 
$A(\Ga)$ by the rule: $(u,v)\in S$  if and only if $\{(u_1,v_2), (v_2,u_1)\}\in M$. 
Then for every vertex $w\in V(\Ga)$, each of the vertices $w_1, w_2$ of $\Delta$
lies in exactly one edge of $M$, and these edges are distinct by the definition of $\Delta$.
Hence there are unique arcs in $S$ of the form $(w,x)$ and of the form $(x,w)$. 
It follows that the sub-digraph $(V(\Ga),S)$ is $1$-regular and spanning. 
Note that each connected component of $(V(\Ga),S)$ is a directed cycle, and 
will have length 2 if and only if its arc set is $\{(u,v),(v,u)\}$ where both 
of the edges $\{(u_1,v_2), (v_2,u_1)\}$ and $\{(v_1,u_2), (u_2,v_1)\}$ of $\Delta$
lie in $M$.
}
\end{proof}



The next lemma considers decompositions of the arc set of a finite simple regular graph. Recall that a simple graph is a simple digraph $(X, A)$ with $A$ symmetric.

\begin{lemma}\label{lem:pdt} 
Let $\Ga=(X,A)$ be a  finite simple  $k$-regular graph with $k\geq 1$, such that, if $k$ is 
odd then $\Ga$ has a perfect matching. Let $k=2n +\delta$ where $\delta\in\{0, 1\}$.
Then $A$ is a disjoint union $\cup_{i=1}^{n+\delta} A_i$  such that, for $1\leq i\leq n$,  $(X, A_i)$ is a  
$2$-regular spanning subgraph, and if $\delta=1$ and $i=n+\delta$, then $(X,A_i)$ 
is a $1$-regular spanning subgraph (a perfect matching).
\end{lemma}

\begin{proof}
{Suppose first that  $\delta=0$ so $k=2n\geq2$. 
By  Petersen's Decomposition Theorem~\cite{MR1554815} (or see \cite[Theorem 6.2.4, p.  218]{lp86}), there is a partition  $A=\cup_{i=1}^{n} A_i$ of the arc set of  $\Gamma$ such that    each  $(X, A_i)$ is a $2$-regular spanning subgraph.  Thus the result holds in this case. 
 Now let $\delta=1$. Then, by assumption,  $\Ga$ has a perfect matching, say $(X, A')$.
 If $k=1$ then $\Ga = (X, A')$ is itself a $1$-regular spanning subgraph and the result holds.
 So suppose that $k=2n+1\geq3$. Then
$\Ga':=(X, A\setminus A')$ is  a simple $(k-1)$-regular graph and $k-1 = 2n$. 
Thus, as we have just seen, $A\setminus A'$ is a disjoint union $\cup_{i=1}^nA_i$  
such that each $(X, A_i)$ is  a  $2$-regular spanning  subgraph. Then, setting 
$A_{n+1} :=A'$, the required result holds for $\Ga$.
}
\end{proof}

Now we apply these two lemmas to prove Theorems~\ref{thm:dareg1}(b) and~\ref{thm:vtr}. Note that Theorem~\ref{thm:dareg1}
follows from Lemmas~\ref{lem:damultfree} and~\ref{lem:daregb}.

\begin{lemma}\label{lem:daregb}
For $k\geq1$, each finite simple $k$-regular digraph $\Ga = (X, A)$ is of the form $\dda(X,S)$ for some $S\subseteq\Der(X)$ with $|S|=k$.
In particular, the assertion of Theorem~\ref{thm:dareg1}(b) is valid. 
\end{lemma}

\begin{proof}{

\medskip\noindent
By Lemma~\ref{lem:dareg}, 
$\Ga$ has a 1-regular spanning sub-digraph $\Ga_1=(X,A_1)$, and such a sub-digraph  
corresponds to a derangement $g_1$ of $X$ as follows: for each $x\in X$, $x^{g_1}$ is the unique vertex 
such that $(x, x^{g_1})$ is an arc of $\Ga_1$.  Since each connected component of  $\Ga_1$ is a 
directed cycle  of length at least 2, it follows that $g_1\in\Der(X)$. 
Now $\Ga':=(X, A\setminus A_1)$ is a simple $(k-1)$-regular digraph, and we recursively 
apply Lemma~\ref{lem:dareg} to $\Ga'$. We obtain a subset $S$ of derangements $g_1,\dots,g_k$ of $X$.
By their definition, the $g_i$ are pairwise distinct, and 
$(x,y)\in A$  if and only if $y=x^{g_i}$ for some $i$. Thus $\Ga=\dda(X,S)$ and $|S|=k$.
}
\end{proof}

\subsection{Proof of Theorem~\ref{thm:vtr} }\label{sec:vtda}

First we deal with parts (a) and (b), so suppose that  $\Ga=(X,A)$ is a finite simple  $k$-regular graph with $k\geq 1$ such that, if $k$ is 
odd, then $\Ga$ has a perfect matching. Let $k=2n +\delta$ where $\delta\in\{0, 1\}$.
By Lemma~\ref{lem:pdt},  $A$ is a disjoint union $\cup_{i=1}^{n+\delta} A_i$  such that, for $1\leq i\leq n$,  $(X, A_i)$ is a  
$2$-regular spanning subgraph, and if $\delta=1$ and $i=n+\delta$, then $(X, A_i)$ is a $1$-regular spanning subgraph
(a perfect matching).

Let $1\leq i\leq n$. Then $(X, A_i)$ is a disjoint union of (undirected) cycles with each 
cycle having length at least $3$. Choose an orientation for each of these cycles 
and let $A_i'$ be the set of arcs (oriented edges) in these cycles. Then $D_i:=(X, A_i')$ is a 
simple $1$-regular spanning sub-digraph of $\Ga$. Also if $\delta=1$ then, 
setting $A_{n+1}':=A_{n+1}$, the sub-digraph $D_{n+1}:=(X, A_{n+1}')$ is a  simple 
$1$-regular spanning subgraph of $\Ga$. For each $i$ such that $1\leq i\leq n+\delta$, 
define $g_i:X\rightarrow X$ by $x^{g_i}=y$ where $(x,y)\in A_i'$. Since $D_i$ is a simple
$1$-regular spanning sub-digraph, it follows that  $g_i$ is well defined and 
is a permutation of $X$ with no fixed points, that is, $g_i\in\Der(X)$. Note that $|g_i|\geq 3$ if $i\leq n$, while if $\delta=1$ then $|g_{n+1}|=2$. 
Let 
\[
S=\{ g_i\mid 1\leq i\leq n+\delta\}\cup\{g^{-1}_i\mid 1\leq i\leq n\}.
\] 
We have just seen that $S\subseteq \Der(X)$ and that $S$ is self-inverse. We will show that 
$S$ is closed, and  that $\Gamma =\da(X, S)$.

By Definition~\ref{def:cds}, to show that $S$ is closed it is sufficient 
to prove (since $S=S^{-1}$) that $SS\subseteq \Der(X)\cup\{1\}$. Let $s, t\in S$. 
If $s=t^{\pm 1}$, then either $st^{-1}=1$, or $st^{-1} = g_j^{\pm 2}$ for some 
$j\leq n$. In the latter case  $g_j^{\pm 2}\in\Der(X)$ since each cycle of 
$D_j$ has length at least $3$. 
Suppose now that $s\ne t^{\pm 1}$. Then $s=g_i^{\pm 1}$ and $t=g_j^{\pm 1}$ for some $i, j$ with $i\ne j$. 
Let $x\in X$.  Then $(x, x^s)\in A_i$ and $(x, x^t)\in A_j$. In particular
$x^s\ne x^t$ and hence $x^{st^{-1}}\ne x$. Since this holds for all $x$, we have $st^{-1}\in\Der(X)$. This proves that $S$ is closed.
 
It now follows from Theorem~\ref{thm:dasimple} that   $\Sigma:=\dda(X,S)$
is a  regular simple graph of valency $|S|$. We claim that $\Sigma=\Ga$. 
By Definition~\ref{def:dad}, $\Sigma = (X, A')$ where $A'=\{ (x,x^s)\mid x\in X, s\in S\}$.
For each $i\leq n+\delta$, we have $A_i'=\{ (x,x^{g_i})\mid x\in X\}$,   and if $i\leq n$ then
$A_i\setminus A_i'=\{ (x,x^{g_i^{-1}})\mid x\in X\}$. Thus $A'=A$ and $\Sigma = \Ga$. 
This completes the proof that each of the graphs in parts (a) and (b) of 
Theorem~\ref{thm:vtr} is a derangement action graph.  

Now we consider part (c). Let  $\Ga=(X,A)$ be a finite simple graph. If $\Ga=(X,A)$ is 
bipartite and $k$-regular, then by \cite[Theorem 3]{np81},  $\Ga$ 
has a perfect matching and the result follows from part (a) if $k$ is even or part (b) if $k$ is odd. 
So we may assume that $\Ga$ is vertex-transitive. Then $\Ga$ is $k$-regular for some 
$k$. The result follows from part (a) if $k$ is even, so assume that $k$ is odd. Let $\Ga'=(X',A')$ be a connected component of 
$\Ga$. Then $\Ga'$ is also a finite simple $k$-regular vertex-transitive graph.  Since $\Ga'$ is a graph,
the number of its arcs is twice the number of its edges, so $|A'|$ is even. Also $|A'|=k|X'|$ 
since $\Ga'$ is $k$-regular, and since $k$ is odd it follows that $|X'|$ is even.  It now
follows from  \cite[Theorem 3.5.1]{god01a} that  $\Ga'$ 
has a perfect matching. Since this holds for all connected components, the graph $\Ga$ has a perfect matching, and the required result now follows from part (b).


\section{Connectivity and products of derangement action digraphs}\label{sec:conn}

In this section we prove Theorem~\ref{thm:connectivity1} and examine various graph product constructions. 
Recall that $x$ is connected to $y$ in $\Ga$ if either $x=y$ or there is a 
vertex sequence $y_0, y_1,\ldots,y_{r}$ such that $y_0=y$, $y_{r}=x$ and $(y_{i-1},y_i)$ is 
an arc for each $i\geq 1$. We refer to such a sequence as a \emph{directed path} from $x$ to $y$ in $\Ga$, and if $x=y$ then we sometimes think of
the one-vertex sequence $x$ as a directed path of length zero.

\medskip\noindent
\textbf{Proof of Theorem~\ref{thm:connectivity1}.}\quad (a)  Let $\Ga = \dda(X,S)$ where $S\subseteq \Der(X)$ 
and each cycle of each element $s\in S$ is finite.
We show that connectivity is an equivalence relation. By definition, each vertex $x$ is connected to itself. Suppose that $x$ is connected to $y$ and $y$ is connected to $z$. Then there is a directed path from $x$ to $y$ and a directed path from $y$ to $z$ (possibly paths of length zero). Concatenating these directed paths we obtain a directed path from $x$ to $z$, so $x$ is connected to $z$. Finally, to prove that the connectivity relation is symmetric, consider first an arc  $(x,y)$  of $\Ga$. Then $y=x^s$ for some $s\in S$. By assumption the cycle of 
$s$ containing $x$ has finite length, say $\ell$,  so $y^{s^{\ell-1}}=x^{s^{\ell}}=x$. Thus, setting $y_i:=y^{s^i}$ for 
$0\leq i\leq \ell-1$, we have a vertex sequence $y_0, y_1,\ldots,y_{\ell-1}$ such that $y_0=y$, $y_{\ell-1}=x$ and $(y_{i-1},y_i)$ is 
an arc for each $i\geq 1$. Thus $y$ is connected to $x$ for each arc $(x,y)$. Now suppose that $u$ is connected to $v$. If $u=v$ then $v$ is connected to $u$ so suppose that $u\ne v$. Then  there is a sequence  $u_0, u_1,\ldots,u_{r}$ such that $u_0=u$, $u_{r}=v$ and $(u_{i-1},u_i)$ is 
an arc for each $i\geq 1$. We have just shown that, for each arc $(u_{i-1},u_i)$ there is a directed path in $\Ga$ from $u_i$ to $u_{i-1}$. Concatenating these paths we obtain a directed path from $v=u_r$ to $u=u_0$, proving that
$v$ is connected to $u$. Thus  connectivity is an equivalence relation.

(b) Assume  that $\Ga = \dda(X,S)$ (with no restrictions on $S$) such that  connectivity is an equivalence relation. 
If $S$ is empty there is nothing to prove, so assume that $S\ne\emptyset$,
Let $x\in X$, let $[x]$ denote the equivalence class containing $x$, and let $X'$ be the 
 $\langle S\rangle$-orbit containing $x$. We claim that $[x]=X'$. 
Let $y\in [x]$. By definition, $x\in X'$ so assume that $y\ne x$. Then there is a directed path  $x_0, x_1,\ldots,x_{r}$ from $x=x_0$ to $y=x_r$. 
 This means that, for each $i<r$, there exists $s_i\in S$ such that $x_i^{s_i}=x_{i+1}$. Hence $x^{s_0\ldots s_{r-1}}=y$ and so $y\in X'$.
 Thus $[x]\subseteq X'$. Conversely let $y\in X'$. Then there exists $t\in \langle S\rangle$ such that $x^t=y$. 
 The element $t$ is a finite product $t=t_0\dots t_r$ where each $t_i=s_i^{\pm 1}$ for some $s_i\in S$. Set $x_0=x$ and $x_{i+1}=x_{i}^{t_i}$ for each $i\geq0$. If $t_i=s_i$ then $(x_i,x_{i+1})$ is an arc of $\Ga$ (a directed path of length 1), while if $t_i=s_i^{-1}$ then $(x_{i+1}, x_i)$ is an arc and we have shown in the proof of part (a)  that there is a directed path from $x_i$ to $x_{i+1}$ in this case also. Concatenating all of these directed paths we obtain a directed path from $x$ to $y$. Thus $y\in [x]$. We conclude that $[x]=X'$ as claimed. 
 Now let $s\in S$, and let $y\in X'$. Then $y^s\in X'$ as $X'$ is an orbit. Also $y^s\ne y$ since
 $s\in\Der(X)$. Thus the restriction $s|_{X'}\in\Der(X')$.  Let $S' := S|_{X'}$ so $S'\subseteq\Der(X')$. Let $\Ga'$ be the sub-digraph induced by $\Ga$ on $X'$. By the definition of $\Ga$, the arcs of $\Ga'$ are precisely the pairs $(x,y)$ such that $x,y\in X'$ and $y=  x^s$ for some $s\in S$. Then $y$ is the image of $x$ under $s|_{X'}$, that is to say, the arcs of $\Ga'$ are precisely the pairs $(x,y)$ from $X'$ with $y$ the image of $x$ under some 
 $s|_{X'}\in S'$. Thus $\Ga'=\dda(X',S')$. 
   
(c)  Assume now that $\Ga$ is a simple digraph such that  connectivity is an equivalence relation. If $\Ga$ is a derangement action digraph then, by part (b), each commented component of $\Ga$ is also a derangement action digraph. Suppose conversely that the connected components are
the derangement action digraphs $\dda(X_i,S_i)$, for $i\in I$, where the $X_i$ form a partition of the vertex set $X$ of $\Ga$ and each $S_i\subseteq \Der(X_i)$. Consider the cartesian product $\mathcal{S} = \prod_{i\in I} S_i$ identified with the set of functions $f:I\rightarrow \cup_{i\in I} S_i$ such that $f(i)\in S_i$ for each $i\in I$. For each $f\in\mathcal{S}$ let $s_f:X\rightarrow X$ be such that, for each $i\in I$ and each $x\in X_i$,  $x^{s_f}=x^{f(i)}$. Then $s_f\in\Der(X)$. Let $S = \{ s_f \mid f\in\mathcal{S}\}$, so $S\subseteq \Der(X)$.   It is straightforward to check that $\Ga = \dda(X,S)$.
\hfill{$\blacksquare$}\medskip

Next we prove that the family of derangement action digraphs is closed under the following four graph products. 
\begin{definition}\label{def:products}
{{\rm Let $\Ga=(X,A)$ and $\Delta=(Y,B)$ be simple digraphs, and let $Z=X\times Y$, and  $z_i=(x_i,y_i)$ for $i=1,2$.\\
(a) The \emph{cartesian product} $\Sigma = \Ga\square \Delta = (Z,C_\square)$, where $(z_1,z_2)\in C_\square$ if and only if  either $(x_1, x_2)\in A$ and $y_1=y_2$, or $(y_1,y_2)\in B$ and $x_1=x_2$.\\
(b) The \emph{tensor product} $\Ga\times \Delta=(Z,C_\times)$, where $(z_1,z_2)\in C_\times$ if and only if  both $(x_1, x_2)\in A$ and $(y_1,y_2)\in B$.\\
(c) The \emph{strong product} $\Ga\boxtimes \Delta=(Z,C_\boxtimes)$, where $C_\boxtimes = C_\square \cup C_\times$. \\
(d)  The  \emph{lexicographic product} $\Ga[\Delta]=(Z,C_{[\ ]})$, where  $(z_1,z_2)\in C_{[\ ]}$ if and only if  either $(x_1, x_2)\in A$, 
or $x_1=x_2$ and $(y_1,y_2)\in B$.
}}
\end{definition}

The direct product $\Sym(X)\times\Sym(Y)$ acts naturally on $X\times Y$ by $(x,y)^{g,h} = (x^g,y^h)$. This action is used to define the derangement subsets of $\Sym(X)\times\Sym(Y)$. For lexicographic products we require a \emph{regular} subgroup $U\leq\Sym(Y)$, that is, $U$ is transitive on $Y$ and $U\setminus\{1_Y\}\subseteq\Der(Y)$. All the assertions are straightforward to check and the details are therefore omitted. 

\begin{theorem}\label{thm:prod1}
{Let $\Ga=\dda(X,S)$ and $\Delta=\dda(Y,T)$ where $S\subseteq\Der(X)$ and $T\subseteq\Der(Y)$, and let $U\leq\Sym(Y)$ act regularly on $Y$. Then
\begin{itemize}
\item[(a)]  $\Ga\square\Delta = \dda(X\times Y,S\square T)$, where $S\square T=(S\times \{1_Y\})\cup(\{1_{X}\}\times T)$;
\item[(b)] $\Ga\times\Delta = \dda(X\times Y,S\times T)$; 
\item[(c)] $\Ga\boxtimes\Delta = \dda(X\times Y,S\boxtimes T)$, where $S\boxtimes T=(S\square T)\cup (S\times T)$;
\item[(d)] $\Ga[\Delta]=\dda(X\times Y, S[T, U])$, where $S[T, U] = (S\times U)\cup (\{1_X\}\times T)$.
\end{itemize}
In particular, $S\square T, S\times T, S\boxtimes T, S[T, U] \subseteq \Der(X\times Y)$, each is closed if both $S$ and $T$ are closed, and each is self-inverse if both $S$ and $T$ are self-inverse.
}\end{theorem}


{\bf Acknowledgements}

The first author would like to thank the School of Mathematics, Institute
for Research in Fundamental Sciences (IPM) for support. This research
was in part supported by a grant from IPM (No. 94050013).





\end{document}